\documentclass[10pt,a4paper]{article}
\usepackage[latin1]{inputenc}
\usepackage{hyperref}
\usepackage{amsmath}
\usepackage{amsthm}
\usepackage{wasysym}
\usepackage{amsfonts}
\usepackage{enumerate}
\usepackage{amssymb}
\author{Randrianarisoa Tovohery Hajatiana}
\title{The number of matrices over $\F$ with irreducible characteristic
polynomial}
\date{}
\newtheorem{lem}{Lemma}

\newtheorem{thm}{Theorem}
\newtheorem{cor}{Corollary}
\newtheorem{defn}{Definition}

\newcommand{\M}{\mathbf{M}_n} 
\newcommand{\GL}{\mathbf{GL}_n} 
\newcommand{\F}{\mathbb{F}_q} 

\begin{document}
\maketitle
\begin{abstract}
Let $\F$ be a finite field with $q$ elements. M. Gerstenhaber
and Irving Reiner has given two different methods to show the number
of matrices with a given characteristic polynomial. In this talk, we
will give another proof for the particular case where the
characteristic polynomial is irreducible. The number of such matrices
is important to know the efficiency of an algorithm to factor
polynomials using Drinfeld modules.
\end{abstract}

In 1960, I. Reiner \cite{Rei60} and M. Gerstenhaber \cite{Ger60} gave two independent methods to show the formula of the number of matrices over a finite field with a given characteristic polynomials. Namely, if we set $\F$ to be the finite field with $q$ elements, then for $g(x)=f_1^{n_1}(x)f_2^{n_2}(x)\cdots f_r^{n_r}(x)$, with each $f_i$ irreducible and distinct to each other, the number of matrices with entries in $\F$ such that $g$ is its characteristic polynomial is,
\[
q^{n^2-n}\frac{F(q,n)}{\prod_{i=1}^{r}F\left(q^{d_i},n_i\right)},
\]
where $n$ is the degree of $g$, $d_i$'s are the degree of $f_i$'s and
\[
F(u,v)=\left(1-u^{-1}\right)\left(1-u^{-2}\right)\cdots \left(1-u^{-v}\right).
\]

Here we are interested only in the case where $g$ is irreducible. That is because this number of matrices is used to compute the efficiency of an algorithm for Factoring polynomials over finite field using Drinfeld modules  \cite{Hei04}. So, first let us restate the simple case we want to prove.

Let $\M$ be the set of $n\times n$ matrices with entries in $\F$, suppose $\GL$ is the set of all the invertible matrices in $\M$.

\begin{thm}\label{thm:main}
Let $f(x)\in \F(x)$ be an irreducible polynomial of degree $n$. Then, the number of matrices in $\M$ with characteristic polynomial $f$ is,
\[
\prod_{i=1}^{n-1}\left(q^n-q^i\right).
\]
\end{thm}

To prove this we need some results in algebra. First of all, we have the Lemma of Schur and its consequence. Assume we have a set $S$ and two representations of $S$ into matrices in $\M$ i.e. we have two maps
\[
s\in S\mapsto R_1(S)\in \M\quad\text{and}\quad s\in S\mapsto R_2(S)\in \M
\]
\begin{defn}
A subset $A$ of $M_n$ is called irreducible if $\lbrace \mathbf{0}\rbrace$ and $\F^n$ are the only invariant subspace of $\F^n$ by elements of $A$.
\end{defn}
\begin{lem}[Schur]
Suppose $R_1$ and $R_2$ are irreducible representations of a set $S$ i.e. $R_1(S)$ and $R_2(S)$ are irreducible subset of $M_n$. Let $M\in \M$ such that, for all $s\in S$, $R_1(s)M=MR_2(s)$. Then either $M=\mathbf{0}$ or $M$ is irreducible. 
\end{lem}
\begin{proof}
We have,
\begin{align*}
R_1(s)(M\M) &= (R_1(s)M)\M\\
&= MR_2(s)\M\\
&= M (R_2(s)\M)\subset M\M.
\end{align*}
Hence, $M\M$ is invariant under $R_1(s)$. By irreducibility, this is either $\mathbf{0}$ or $M_n$. The first case gives us $M=\mathbf{0}$. For the second case, assume $K$ is the kernel of $M$ as endomorphism of $\F^n$. Then,
\begin{align*}
M(R_2(s)K)&= (MR_2(s))K\\
&= (R_1(s)M)K\\
&= R_1(s)(MK)\\
&= R_1(s)(\mathbf{0})\\
&= \mathbf{0}.
\end{align*}
Hence, $K$ is invariant under $R_2(s)$. By irreducibility, we can only have $K=\mathbf{0}$. Thus $M$ is bijective so that the matrix is invertible.
\end{proof}
\begin{cor}\label{cor:1}
If $A$ is an irreducible subset of $M_n$ then, between centralizers, we have
\[
C_{\GL}(A)=C_{\M}(A)-\mathbf{0}
\]
\end{cor}
\begin{proof}
In the previous lemma, we take $R_i(S)=A$ and by definition elements of centralizers satisfy $Ma=aM$, for all $a\in A$.
\end{proof}

Next result we need is about rational canonical form. Let $V$ be a finite dimensional vector space over a field $\F$. Let $M$ be a linear operator on $V$ and suppose $v$ is a vector in $V$.

\begin{itemize}
\item There exist a non-zero polynomial in $\F[x]$, $p_v(x)$ such that $p_v(M)v=0$. The minimal monic polynomial with such property is called the order of $v$. This is the minimal polynomial for the endomorphism $M$ restricted to $<v>$.
\item There exist $\lbrace v_i\rbrace $ in $V$, such that 
\[
V=<v_1>\oplus \cdots \oplus <v_r>,
\]
and the order $p_{i}$ of the $v_i$ are power of prime polynomials in $\F[x]$.
\item Finally, if $C_i$ are the companion matrix of $p_i$, then, $M$ is similar to the diagonal block matrix $\left(C_i\right)$. This new matrice is called the \textit{canonical rational form} of $M$.
\end{itemize}
\begin{lem}
Two matrices $A$ and $B$ are conjugate if they have the same canonical form (after rearranging blocks).
\end{lem}
\begin{proof}
We just use the fact that these matrices are conjugate to their rational canonical form.
\end{proof}
\begin{cor}\label{cor:2}
All matrices in $\M$ with the same irreducible characteristic polynomials are conjugate.
\end{cor}
\begin{proof}
Since the characteristic polynomial is irreducible, there can only be one block in the canonical rational form.
\end{proof}

Finally, the last result we need, is taken from \cite{ST}.
\begin{lem}\label{lem:3}
Let $M$ be a matrix in $\M$. The centralizer of $M$ in $\M$ is equal to $\F\left[M\right]$ if and only if the characteristic polynomial of $M$ is irreducible. 
\end{lem}
Note that this equivalence is essentially due to the fact that $\F$ is a finite field and $M$ has irreducible characteristic polynomial.

We are now ready to prove the main theorem.

\begin{proof}[Proof of theorem \ref{thm:main}]
We operate $\GL$ on $\M$ by conjugation. Then for a matrix $M$, with irreducible characteristic polynomial $f$, we have
\[
\left|Orb\left(M,\GL\right)\right|=\left[ \GL:Stab\left(M,\GL\right) \right].
\]
Now, the stabilizer of this operation is $Stab\left(M,\GL\right)=C_{\GL}(M)$. But since, $M$ has irreducible characteristic polynomial, then $\lbrace M\rbrace$ is irreducible in $\M$. Thus, by the corollary \ref{cor:1}, $Stab\left(M,\GL\right)=C_{\M}(M)-\mathbf{0}$. By lemma \ref{lem:3}, $\left|C_M(M)\right|=q^n$. Therefore, $\left|Stab\left(M,\GL\right)\right|=q^n-1$.
Now, corollary \ref{cor:2} tells us that the number of matrices with the same irreducible characteristic polynomial is $\left|Orb(M\GL)\right|$. Finally the number of invertible matrices is $\prod_{k=0}^{d-1}(q^d-q^k)$ and the result follows.
\end{proof}
\bibliographystyle{plain}
\bibliography{article}
\end{document}